\documentclass[12pt]{amsart}
\usepackage{amsmath,amssymb,amsthm}
\usepackage{geometry}                
\usepackage{graphicx}
\usepackage{amssymb}
\usepackage{epstopdf}
\usepackage{tikz}
\usepackage{enumerate}
	\def\set#1{\left\{ {#1} \right\}}
	\def\setof#1#2{{\left\{#1\,:\,#2\right\}}}
	
	\def\P{{\mathbf P}}
	
	\def\V{{\tilde{V}}}
	
	\def\H{{\mathcal H}}
	\def\LL{{\mathbb L}}

	\def\Ass{{\text{Ass}}}

	\def\isomorphic{{\,\cong\,}}

\theoremstyle{plain}

\newtheorem{thm}{Theorem}[section]
\newtheorem{cor}[thm]{Corollary}
\newtheorem{prop}[thm]{Proposition}
\newtheorem{lem}[thm]{Lemma}

\theoremstyle{definition}

\newtheorem{ex}[thm]{Example}

\newtheorem{defn}[thm]{Definition}

\newtheorem{conj}[thm]{Conjecture}
\newtheorem{question}[thm]{Question}
\newtheorem{notation}[thm]{Notation}
\newtheorem{rem}[thm]{Remark}

\title{On the Fattening of Lines in $\P^3$}
\author{Mike Janssen}
\address{Department of Mathematics, University of Nebraska, Lincoln, Nebraska 68588-0130}
\email{janssen@huskers.unl.edu}
\thanks{The author wishes to thank Brian Harbourne for many helpful conversations during the preparation of this work, especially in the process of writing the dissertation \cite{Janssen}. The author also wishes to thank Juan Migliore for his very helpful comments on arithmetically Cohen Macaulay schemes, and Tomasz Szemberg for his very helpful suggestion for a simpler proof of Proposition \ref{PropGenHyperplaneNon}.}

\begin{document}

\maketitle

\begin{abstract}
We follow the lead of \cite{BocciChiantini} and show how differences in the invariant $\alpha$ can be used to classify certain classes of subschemes of $\P^3$.
	Specifically, we will seek to classify arithmetically Cohen-Macaulay codimension 2 subschemes of $\P^3$ in the manner Bocci and Chiantini classified points in $\P^2$. 
	The first section will seek to motivate our consideration of the invariant $\alpha$ by relating it to the Hilbert function and $\gamma$, following the work of \cite{BocciChiantini,Dumnicki2}.
	The second section will contain our results classifying arithmetically Cohen-Macaulay codimension 2 subschemes of $\P^3$.
	This work is adapted from the author's Ph.D.\ dissertation \cite{Janssen}.
\end{abstract}


	\section{The Importance of $\alpha$}
		Let $k$ be an algebraically closed field of arbitrary characteristic.
		Much is known about finite sets of reduced points $Z\subseteq \P^2$ over $k$.
		In particular, \cite{GMR} classified all possible Hilbert functions of finite sets of reduced points in $\P^N$ over $k$.
		However, not much is known about the double scheme, $2Z$ (however, see \cite{GMS,GHM}).
		
		\begin{defn}\label{DefSymbolicPower}
			In general, the $m$-th symbolic power of a homogeneous ideal $I\subseteq R=k[\P^N]$ is 
			\[
				I^{(m)} = R \cap \left( \cap_{P\in\Ass(I)} (I^m R_P) \right),
			\]
			where $\Ass(I)$ denotes the set of associated primes of $I$ and $R_P$ is the ring $R$ localized at the prime $P$.
		\end{defn}
		
		When $I$ is the ideal of a complete intersection, $I^{(m)} = I^m$ (see \cite[Lemma 5, Appendix 6]{Zariski-Samuel}), and thus if $I = (L_1,L_2)$ (where $L_1,L_2$ are linear forms) is the ideal of a linear codimension 2 complete intersection (e.g., a point in $\P^2$, or a line in $\P^3$), we have $I^{(m)} = (L_1,L_2)^m$.

		\begin{defn}\label{DefDoubleScheme}
			Let $Z\subseteq \P^N$ be a reduced subscheme defined by $I = I(Z)$.
			The double scheme (often called the double point scheme if $Z$ is a set of reduced points, or the fattening) is the subscheme of $\P^N$ defined by $I^{(2)}$ and denoted $2Z$.
		\end{defn}
		
		In this paper, we follow the lead of Bocci and Chiantini \cite{BocciChiantini} and others in studying the number $\alpha(I)$, where $I$ is the ideal of a reduced codimension 2 subscheme.
		Recall that, if $I$ is a nonzero homogeneous ideal in $k[\P^N]$, the number $\alpha(I)$ is the degree of a nonzero polynomial of least degree in $I$.
		(Equivalently, if $I_d$ denotes the homogeneous component of $I$ of degree $d$, $\alpha(I) = \min\setof{d}{I_d\not=0}$.)
		
		%
		%
		
		Now, it is not difficult to see that the number $\alpha$ is the degree in which the Hilbert function of the quotient $R/I$ first deviates from that of the ring $R=k[\P^N]$.
		Indeed, recall that the Hilbert function of the quotient of a homogeneous ideal $I\subseteq R = $ in degree $t$ is $H(R/I,t) = \dim_k (R_t) - \dim_k (I_t)$.
		If $t < \alpha(I)$, $\dim_k (I_t) = 0$, hence $H(R/I,t) = \dim_k (R_t) = \binom{t+N}{N}$.
		
		Thus, Bocci and Chiantini, rather than compute $\alpha$ (or even Hilbert functions) of various planar point configurations in $\P^2$ or their symbolic powers, chose to study the difference $t:= \alpha(2Z) - \alpha(Z)$.
		This is related to the Waldschmidt constant 
		\[
			\gamma(I) := \lim\limits_{m\to\infty} \frac{\alpha(I^{(m)})}{m}.
		\]
		Understanding $\alpha(I^{(m)})$ for all $m\geq 1$ is a difficult task, while classifying $Z$ based on the difference $t$ is more tractable (though, as $t$ increases, it grows more difficult).
		
		An important first observation about $t$ is that $t \geq 1$ always holds; indeed, if $k$ has characteristic 0, let $F$ be a form of minimal degree $\alpha(2Z)$ vanishing to order at least 2 at each point of $Z$.
		Then the partial derivatives of $F$ vanish on $Z$, and the degree of the partial derivatives is less than the degree of $F$.
		If $k$ has characteristic $p > 0$, then it may happen that every partial derivative of $F$ is identically 0.
		In that case, $F$ is the $p$th power of some form $G$, which vanishes at each point of $Z$, and thus $t \geq 1$.
		
		We follow the lead of \cite{BocciChiantini} and say that a subscheme $Z\subseteq \P^N$ has type $(d-t,d)$ if $\alpha(Z) = d-t$ and $\alpha(2Z) = d$.
		
		Bocci and Chiantini examine cases when $t$ is small; specifically, they consider $t=1,2$.
		When $t=1$, they use B\'{e}zout's Theorem to find:
		
		\begin{thm}[Example 3.1, Proposition 3.2, and Theorem 3.3 of \cite{BocciChiantini}]\label{TheoremBC}
			Let $Z\subseteq \P^2$ be a finite set of points.
			Then $t = 1$ if and only if either $Z$ is a set of collinear points and $\alpha(Z) = 1$ or $Z$ is a star configuration of $\binom{d}{2}$ points and $\alpha(Z) = d-1$.
		\end{thm}
		That $\alpha(Z) = 1$ and $\alpha(2Z) = 2$ when $Z$ is a set of collinear points is clear.
		A star configuration of points in $\P^2$ is the finite subset $Z$ of ${d\choose 2}$ points of pairwise intersection of $d$ lines, where $d\geq 3$.
		See Figure \ref{FigStarConfig2} for a star configuration $Z$ when $d=5$.
		Let $F$ be the product of the five linear forms corresponding to the lines, and $G$ be the form $F$ divided by one of the linear forms.
		Then it is clear that $F$ vanishes to order 2 at each of the 10 points and $G$ vanishes to order at least 1 at each point; it is known that $F$ and $G$ are forms of minimal degree vanishing to order 2 and 1, respectively, which means that $Z$ has type $(4,5)$.

		\begin{figure}\label{FigStarConfig2}
			\centering
			\begin{tikzpicture}[scale=1]
				\filldraw [black] 
				(11/13,6/13) circle (2pt)
				(5/3,2/3) circle (2pt)
				(9/7,-6/7) circle (2pt)
				(4/3,4/3) circle (2pt)
				(16/9,10/9) circle (2pt)
				(2/5,9/5) circle (2pt)
				(7/3,5/6) circle (2pt)
				(3/4,3/4) circle (2pt)
				(2,2) circle (2pt)
				(1/3,1/3) circle (2pt);
				\draw (-1,-1) -- (3.1,3.1);
				\draw (-1,2.5) -- (3,.5); 
				\draw (0,3) -- (4.2/3,-1.2) ;
				\draw (-1,0) -- (3,1) ;
				\draw (1.2,-1.2) -- (9/4,3); 
			\end{tikzpicture}
			\caption{A star configuration formed by the pairwise intersection of 5 lines in $\P^2$.}
		\end{figure}
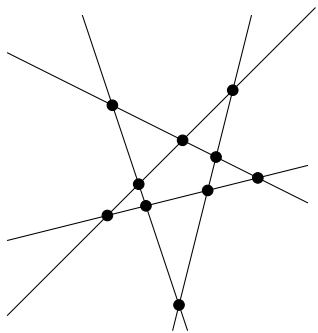


		When $t=2$, Bocci and Chiantini also obtain classification results, though these are much more complicated. 
		The situation can be roughly described as follows: either $\alpha(2Z) = 4$ and $Z$ lies in a conic, or $\alpha(2Z) > 4$ and $Z$ lies in the nodes of the union of rational curves.
		
		There are several possible avenues for generalizing these results; the first we consider is to look at higher symbolic powers.
		We borrow the following notation from \cite{Dumnicki2}:
		\begin{notation}
			Let $Z\subseteq \P^2$ be a finite fixed set of arbitrary points.
			Then we use the notation $\alpha_{m,n}(Z):= \alpha(I^{(m)}) - \alpha(I^{(n)})$ for $m > n$.
		\end{notation}
		In \cite{Dumnicki2}, Dumnicki et al.\ obtain stronger results by requiring the successive differences $\alpha_{m+1,m}$ to be constant as $m$ increases.

		They then prove:
		
		\begin{thm}[Theorems 3.1 and 4.14 of \cite{Dumnicki2}]\label{ThmDumnMain}
			If
			\[
				\alpha_{2,1}(Z) = \alpha_{3,2}(Z) = \cdots = \alpha_{t+1,t}(Z) = d,
			\]
			then
			\begin{enumerate}
				\item for $d=1$ and $t \geq 2$ the set $Z$ is contained in a line, i.e., $\alpha(Z) = 1$;
				\item for $d=2$ and $t \geq 4$ the set $Z$ is contained in a conic, i.e., $\alpha(Z) = 2$.
			\end{enumerate}
			Moreover, both results are sharp, i.e., there are examples showing that one cannot relax the assumptions on $t$.
		\end{thm}
		The authors believe that such a result should be true for cubics as well.
		
		Another recently-explored avenue is to points in $\P^1\times\P^1$; in \cite{PntsFttnngInP1xP1}, the authors extend the results of \cite{BocciChiantini,Dumnicki2} to bi-homogeneous ideals over $\P^1\times \P^1$.
		
		A third avenue for generalizing the results of Bocci and Chiantini is to consider subschemes of higher dimensional projective spaces, and this is the direction we will take in the remainder of this note. 
		However, rather than look at point configurations, we will examine configurations of lines in $\P^3$.
		With some additional reasonable assumptions, we are able to reduce to Bocci and Chiantini's results to describe configurations of lines in $\P^3$ for which $t=1$.

	\section{Lines in $\P^3$}

		Throughout the remainder, let $S = k[\P^3] = k[x,y,z,w]$ and $R = k[\P^2] = k[x,y,z]$ be the homogeneous coordinate rings of $\P^3$ and $\P^2$, respectively.

		Broadly speaking, the two types of configurations of lines in $\P^3$ we will discuss are the coplanar configurations and the pseudo-star configurations.

		\begin{defn}
			A pseudo-star configuration (or pseudostar) of lines in $\P^3$ is a finite collection of lines formed by the pairwise intersection of hyperplanes such that no three of the hyperplanes meet in a line.
		\end{defn}

		There is a growing body of literature on the study of star configurations (see \cite{GHM} and the references therein).
		Indeed, star configurations were one of the first examples studied in \cite{BocciHarbourne1} in which the resurgence $\rho(I)$ was introduced.
		The easiest examples, of course, are star configurations of points in $\P^2$, but star configurations can be defined in any codimension in any projective space.
		
		%
		%
		
		As defined in \cite{GHM}, a star configuration of lines in $\P^3$ is a collection of lines formed by the pairwise intersections of hyperplanes which meet properly, meaning that the intersection of any $j$ of the hyperplanes is empty or has codimension $j$.
		For the case of the pseudostars, we replace the requirement that the planes meet properly with the requirement that no three of the planes meet in a line; therefore, it may be that in a pseudostar in $\P^3$, more than three planes meet in a single point.
		
		The easiest example of a pseudostar in $\P^3$ is a star configuration of lines.

		Another easy example of a pseudostar in $\P^3$ is a projective cone over a star configuration of points in $\P^2$:

		\begin{ex}
			Suppose $I\subseteq R$ defines a star configuration $Z$ of points in $\P^2$.
			The projective cone over $Z$ is a subscheme of $\P^3$ defined by the extension $IS$ of $I$ to $S$.
			This is an example of a pseudostar.
		\end{ex}
		
		The proof of our main theorem will be powered by the notion of arithmetically Cohen-Macaulay subschemes.
		In particular, for an ACM subscheme $X\subseteq \P^N$ we will exploit the relationship between $\alpha(X)$ and $\alpha(X\cap H)$, where $H$ is a general hyperplane.
		%
		\begin{defn}
			A subscheme $X\subseteq \P^N$ is arithmetically Cohen-Macaulay (ACM) if the homogeneous coordinate ring $k[\P^N]/I(X)$ of the subscheme is Cohen-Macaulay.
		\end{defn}

		Several familiar linear configurations are ACM.
		
		\begin{lem}
			Any collection of coplanar lines in $\P^3$ is ACM.
		\end{lem}
		
		\begin{proof}
			If $I\subseteq S$ is the ideal of coplanar lines, then $I$ is a complete intersection ideal, and thus $S/I$ is Cohen-Macaulay.
		\end{proof}

		\begin{lem}
			Let $\LL$ denote a finite union of lines in $\P^3$.
			If $\LL$ is a star configuration of lines in $\P^3$ or a projective cone over a star configuration of points in $\P^2$, $\LL$ and $2\LL$ are ACM.
		\end{lem}

		\begin{proof}
			If $\LL$ is a star configuration of lines in $\P^3$, then $\LL$ and $2\LL$ are ACM by \cite[Proposition 2.9 and Theorem 3.1]{GHM}, respectively.
			Suppose $\LL$ is a projective cone over a star configuration $Z$ in $\P^2$.
			Then $I(\LL) = I(Z)S$, and $(R/I(Z))[w] \isomorphic S/I(Z)S = S/I(\LL)$.
			Since $R/I(Z)$ is Cohen-Macaulay, so is $(R/I(Z))[w]$, and hence also $S/I(\LL)$. 
			Therefore $\LL$ is ACM. 
			A similar argument can be carried out for $I(2Z) = (I(Z))^{(2)}$.
		\end{proof}




		\begin{prop}\label{PropPseudostarsACM}
			Pseudostars and their symbolic squares are ACM.
		\end{prop}

		\begin{proof}
			The reduced case was proved, though not explicitly, in \cite[Proposition 2.9]{GHM} (but see \cite[Remark 2.13]{GHM}).
			The symbolic square case can be found in the first part of the proof of \cite[Theorem 3.2]{GHM}, as the assumption that the hyperplanes meet properly can be relaxed to the assumption that no three hyperplanes contain a line.
		\end{proof}

		\begin{prop}[Corollary 1.3.8 of \cite{MiglioreLiaisonTheoryBook}]\label{PropHypSecsEqualAlphas}
			Let $X\subseteq \P^N$ be an arithmetically Cohen-Macaulay scheme of dimension at least 1, and suppose $H\subseteq \P^N$ is a general hyperplane.
			Let $X\cap H$ denote the general hyperplane section of $X$, $S=k[\P^N]$, and $R = S/I(H) \,\isomorphic\, k[\P^{N-1}]$.
			Then the Hilbert function of $R/I(X\cap H)$ is given by
			\[
			H(R/I(X\cap H),t) = H(S/I(X),t) - H(S/I(X),t-1).
			\]
		\end{prop}

		A useful corollary of Proposition \ref{PropHypSecsEqualAlphas} is the following.
		
		\begin{cor}\label{CorHypSecsEqualAlphas}
			Suppose $X\subseteq \P^N$ is an arithmetically Cohen-Macaulay scheme of dimension at least 1, and $H\subseteq \P^N$ is a general hyperplane.
			If $X\cap H$ denotes the general hyperplane section of $X$, then $\alpha(X) = \alpha(X\cap H)$.
		\end{cor}
		\begin{proof}
			This follows immediately from Proposition \ref{PropHypSecsEqualAlphas} and the definitions of the Hilbert function and $\alpha$.
		\end{proof}

		\begin{cor}\label{CorPseudoStarEqualAlpha}
			Let $\LL$ be a pseudostar in $\P^3$ formed by the pairwise intersection of $d$ planes, no three of which contain any line.
			Then $\alpha(\LL) = d-1$ and $\alpha(2\LL) = d$.
		\end{cor}

		\begin{proof}
			We first fix our notation.
			Let $H_1,H_2,\ldots,H_d\subset \P^3$, $d > 2$ (if $d \leq 2$, the lines resulting from the pairwise intersection of the hyperplanes will be coplanar) be hyperplanes, no three of which contain any line.
			Set $\ell_{ij} = H_i\cap H_j$ for all $i < j$, and put $\LL = \bigcup\limits_{1\leq i < j \leq d} \ell_{ij}$.
			Then $\LL$ is a pseudostar.
			
			We first show that $\alpha(\LL) = d-1$.
			By Corollary \ref{CorHypSecsEqualAlphas}, it is enough to show that the general hyperplane sections of $\LL$ form a star configuration of points in $\P^2$.
			
			A general hyperplane $H$ meets each $H_i$ in a line $L_i$; as $H$ is general, $L_i$ meets each $\ell_{ij}$, $j\not=i$ in distinct points $p_{ij}\in H \isomorphic \P^2$.
			The points $p_{ij}$, $j\not=i$, form a star configuration of points in $H\isomorphic \P^2$, as each line $L_i$ contains $d-1$ points $p_{ij}$, $j\not=i$; each point $p_{ij}$ lies on exactly two lines, $L_i$ and $L_j$, hence we have exactly ${d\choose 2}$ points.
			Thus, the general hyperplane sections of $\LL$ form a star configuration.
			
			To see that $\alpha(2\LL) = d$, note that $d\geq \alpha(2\LL) = \alpha(2(\LL\cap H)) > \alpha(\LL\cap H)  = d-1$.
			
			%
		\end{proof}

			The following proposition shows that if a general hyperplane intersects three or more lines in $\P^3$ in collinear points, the lines must lie in a plane.
			We make use of the notion of the dual space of $\P^3$, which we denote $(\P^3)^*$.
			Recall the dual relationship: a point $(a,b,c,d)\in \P^3$ corresponds to a hyperplane $ax+by+cz+dw=0$ in $(\P^3)^*$.

		\begin{prop}\label{PropGenHyperplaneNon}
			A general hyperplane intersects $d\geq 3$ non-coplanar lines in $\P^3$ in $d$ non-collinear points.
		\end{prop}
		\begin{proof}
			Let $\LL$ be the union of $d = 3$ non-coplanar lines in $\P^3$ (if $d > 3$, then there are three non-coplanar lines in the union $\LL$, so we choose to focus on those three lines).
			Note that the condition that the intersection points are collinear is a closed condition (i.e., the set of planes $ax+by+cz+dw=0$ meeting the lines in collinear points yields a closed subset of points $(a,b,c,d) \in (\P^3)^*$), so it is enough to show that these are not all the planes; that is, if there exists a plane which meets $\LL$ in non-collinear points, then the set of all such planes must be a nonempty open subset of $(\P^3)^*$.

			Assume not; that is, assume every plane meets $\LL$ in collinear points.
			Let $p_1 \in \ell_1$ such that $p_1 \not\in \ell_2 \cup \ell_3$ and $p_2 \in \ell_2$ such that $p_2 \notin \ell_1 \cup \ell_3$.
			Given points $q,r\in \P^3$, let $L(q,r)$ denote the line through $q$ and $r$.
			Then we must have $L(p_1,p_2) \cap \ell_3 \not=\emptyset$ (otherwise any plane containing $L(p_1,p_2)$ would meet $\ell_1\cup \ell_2\cup\ell_3$ in non-collinear points).
			Next, pick $p_2' \in \ell_2$ such that $p_2'\not=p_2$ and $p_2'\notin \ell_1 \cup \ell_3$.
			Then $L(p_1,p_2')$ also meets $\ell_3$ (in such a way that $L(p_1,p_2)\cap \ell_3 \not= L(p_1,p_2') \cap \ell_3$), and thus the plane $H_1$ formed by $L(p_1,p_2)$ and $L(p_1,p_2')$ contains $\ell_2$ and $\ell_3$.
			
			If $\ell_1$ is contained in $H_1$, we have a contradiction, as the lines are assumed to be non-coplanar.
			Thus, assume $\ell_1$ is not contained in $H_1$; choose $q_1 \in \ell_1$ such that $q_1\notin H_1$ and let $q_1'$ denote the unique point of intersection between $H_1$ and $\ell_1$.
			Furthermore, let $q_2 \in \ell_2$ such that $q_2\notin \ell_1 \cup \ell_2$, and $q_3 \in \ell_3$ such that $q_3 \notin \ell_1 \cup \ell_2 \cup L(q_1',q_2)$.
			Clearly, $q_1,q_2,q_3$ are coplanar, but they are not collinear, as $q_1 \notin H_1$ and $L(q_2,q_3) \subseteq H_1$.
			Thus, the plane formed by $q_1, q_2$, and $q_3$ does not meet the lines $\ell_1,\ell_2,\ell_3$ in collinear points.
		\end{proof}

		Another way to say this is:

		\begin{cor}\label{CorCoplanarLines}
			If $d \geqslant 3$ lines in $\P^3$ intersect a general hyperplane $H$ in collinear points, then the lines are coplanar.
		\end{cor}

		We set the following notation.

		\begin{notation}
			Let $H_1,H_2,\ldots,H_d\subset \P^3$ be hyperplanes, no three of which contain any line.
			Set $\ell_{ij} = H_i\cap H_j$ for all $i < j$, and put $\LL = \bigcup\limits_{1\leq i < j \leq d} \ell_{ij}$.
		\end{notation}

		We now come to the main result of this work, which describes an extension of Bocci and Chiantini's $t=1$ result for points in $\P^2$.
		In $\P^2$ every codimension 2 subscheme is ACM, as all finite sets of points in any $\P^N$ are ACM.
		In higher dimensions, not every codimension 2 subscheme in $\P^N$ is ACM (e.g., three skew lines in $\P^3$).
		However, the natural generalization of the Bocci-Chiantini result seems to be for ACM codimension 2 subschemes (but see Question \ref{QuestionExist1notACM}).


		\begin{thm}\label{ThmACMLines}
			Let $\LL$ be a union of lines $\ell_1,\ell_2,\ldots,\ell_s$.
			\begin{enumerate}[(a)]
				\item If $\LL$ is ACM of type $(d-1,d)$ for some $d>1$, then $\LL$ is either a pseudostar or coplanar.
				\item If $\LL$ is either a pseudostar or coplanar, then $\LL$ has type $(d-1,d)$ for some $d>1$.
			\end{enumerate}

		\end{thm}
		

		\begin{proof}[Proof of Theorem \ref{ThmACMLines}]
			We begin by proving (a), and first treat the cases in which $1\leq s \leq 3$ in an ad hoc fashion.
			
			 Indeed, if $s=1$, we have a single line, which is coplanar, so (a) holds.
			
			For $s=2$, either the lines meet, in which case they are coplanar, or the lines are skew.
			If the lines $\ell_1,\ell_2$ are skew, then, without loss of generality, we may take $I(\ell_1) = (x,y)$ and $I(\ell_2) = (z,w)$, so $I(\LL) = (x,y)\cap(z,w)$, $\alpha(\LL) = 2$, and $\alpha(2\LL) = \alpha((x,y)^2\cap(z,w)^2) = 4$ so $\LL$ has type $(d-2,d)$.
			In either case, if $s=2$, (a) holds.
			
			If $s=3$, we have three possible configurations.
			If the lines meet in a single point, they are either coplanar or a pseudostar.
			If the lines do not meet in a single point but intersect pairwise, they are coplanar.
			\begin{figure}\label{Fig2SkewLines}
				\centering
					\begin{tikzpicture}[scale=5/6]
						\filldraw [black] (-1/2,0) circle (2pt) node[below] {}
						(3/2,0) circle (2pt) node[below] {};
						\draw (-3/2,0) -- (5/2,0) node {\,\,\,\,\qquad $\ell_1$};
						\draw (-1/2,-1) -- (-1/2,3) node {\!\!\!\!\!\!\! $\ell_2$}; 
						\draw (3/2,-1) -- (3/2,3) node {\,\,\,\,\,\,\,$\ell_3$}; 
					\end{tikzpicture}
				\caption{Three lines $\ell_1,\ell_2,\ell_3$ in $\P^3$ such that $\ell_2\cap\ell_3 = \emptyset$.}
			\end{figure}
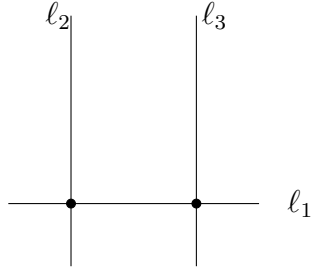
			The last case involves lines $\ell_1,\ell_2,\ell_3$ such that $\ell_2$ and $\ell_3$ do not meet, but $\ell_2\cap \ell_1 \not=\emptyset$ and $\ell_3\cap\ell_1 \not= \emptyset$, as in Figure \ref{Fig2SkewLines}.
			In this case, we can, after an appropriate change of coordinates, assume $I(\ell_1) = (x,z)$, $I(\ell_2) = (y,z)$, and $I(\ell_3) = (x,w)$.
			One can easily verify that $\LL = \ell_1\cup \ell_2\cup \ell_3$ is of type $(2,4)$.
			Thus, (a) is satisfied for $1\leq s \leq 3$. To finish the proof we consider the case that $s \geq 4$. 

			Suppose $\LL$ has type $(d-1,d)$ for some $d\geq 2$, and let $H$ denote a general hyperplane.
			As $\LL$ is ACM, we can apply Proposition \ref{PropHypSecsEqualAlphas} to $\LL$ to see that $\alpha(\LL) = \alpha(\LL\cap H) = d-1$, and since $d = \alpha(2\LL) \geq \alpha(2(\LL\cap H)) > \alpha(\LL\cap H) = d-1$ (see \cite{BocciChiantini}), the general hyperplane sections $\LL\cap H$ must have type $(d-1,d)$ in $H\isomorphic \P^2$.
			By \cite{BocciChiantini}, this means that the general hyperplane sections $\LL\cap H$ of $\LL$ are either a set of collinear points or a star of points in $\P^2$.

			If $\LL\cap H$ is a set of collinear points, we must have that $\LL$ is a set of coplanar lines (see Proposition \ref{PropGenHyperplaneNon}).
			Otherwise, by Proposition \ref{PropGenHyperplaneNon} (since $s\geq 4$ and thus $s-1 \geq 3$) we have $d$ (non-disjoint) collections of $d-1$ collinear points (in fact, we have ${d\choose 2}$ points total, since $\LL\cap H$ is a star in $H\isomorphic \P^2$).
			Each of the ${d\choose 2}$ points is the hyperplane section of exactly one of the $\ell_{ij}$'s, so we must have $s ={d\choose 2}$ lines $\ell_{ij}$, with $d$ (non-disjoint) collections of $d-1\geq 3$ coplanar lines.
			Moreover, since we have $d$ hyperplanes meeting in ${d\choose 2}$ lines, it must be that no three hyperplanes meet in a line, or else we would have strictly fewer than ${d\choose 2}$ lines, and thus strictly fewer than ${d\choose 2}$ hyperplane sections.
			Thus, $\LL$ forms a pseudostar.

			We now turn to (b). 
			Note that if $\LL$ is coplanar, $\LL$ clearly has type $(1,2)$.
			If $\LL$ is a pseudostar, then by Corollary \ref{CorPseudoStarEqualAlpha} $\LL$ has type $(d-1,d)$ for some $d> 1$.
			
		\end{proof}

	\section{Future Work}
		It seems as that pairing this approach with an inductive argument may generalize Theorem \ref{ThmACMLines} to ACM codimension 2 subschemes of $\P^N$, $N > 3$, but this has not yet been explored.
		
		There are several other avenues for future work.

		We made heavy use of the assumption that the lines in question in $\P^3$ are arithmetically Cohen-Macaulay (ACM).
		A natural question, then, is:
%
		%
		%
		\begin{question}\label{QuestionExist1notACM}
			Does there exist a configuration of lines of type $(d-1,d)$ which is not ACM?
		\end{question}
		
		If the answer to Question \ref{QuestionExist1notACM} is no, then the ACM hypothesis in Theorem \ref{ThmACMLines} is unnecessary. 
		
		Another interesting question (suggested by Juan Migliore) is:
		
		\begin{question}
			Which reduced (possibly irreducible) curves in $\P^3$ have type $(d-1,d)$ for some $d>1$?
		\end{question}

%
		As every finite set of points in $\P^N$ (for any $N\geq 1$) is ACM, another natural question to ask is:

		\begin{question}
			Which configurations of points in $\P^3$ have type $(d-1,d)$?
		\end{question}

		In \cite{BocciChiantini}, the authors also classify configurations of points in $\P^2$ which have type $(d-2,d)$.
		Thus, we ask:

		\begin{question}
			Which arrangements of lines in $\P^3$ have type $(d-2,d)$?
			Which arrangements of points in $\P^3$ have type $(d-2,d)$?
		\end{question}

\bibliography{thesis}{}
\bibliographystyle{plain}

\end{document}